\documentclass[11pt]{amsart}

\usepackage{palatino}
\usepackage{amsfonts}
\usepackage{amssymb}
\usepackage{amscd}
\usepackage{xypic}
\usepackage[breaklinks,bookmarksopen,bookmarksnumbered]{hyperref}
\usepackage{graphicx}
\usepackage{float}

\newtheorem{theorem}{Theorem}[section]
\newtheorem{proposition}[theorem]{Proposition}
\newtheorem{corollary}[theorem]{Corollary}

\theoremstyle{definition}
\newtheorem{definition}[theorem]{Definition}
\newtheorem{remark}[theorem]{Remark}

\newtheorem{conjecture/question}[theorem]{Conjecture/Question}

\newtheorem{remark/definition}[theorem]{Remark/Definition}
\newtheorem{terminology/notation}[theorem]{Terminology/Notation}

\def\GG{{\textbf G}}

\def\PP{{\textbf P}}
\def\FF{{\textbf F}}
\def\OO{\mathcal{O}}

\def\F{\mathcal{F}}
\def\P{\mathcal{P}}

\def\cC{\mathcal{C}}
\def\H{\mathcal{H}}
\def\hh{\mathfrak{H}_{\mathrm{scr}}}

\def\mm{\overline{\mathcal{M}}}

\DeclareFontFamily{OT1}{pzc}{}
\DeclareFontShape{OT1}{pzc}{m}{it}{<-> s * [1.10] pzcmi7t}{}
\DeclareMathAlphabet{\mathpzc}{OT1}{pzc}{m}{it}

\begin{document}
\title{The unirationality of the moduli space of $K3$ surfaces of genus 22}

\author[G. Farkas]{Gavril Farkas}

\address{Humboldt-Universit\"at zu Berlin, Institut f\"ur Mathematik,  Unter den Linden 6
\hfill \newline\texttt{}
 \indent 10099 Berlin, Germany} \email{{\tt farkas@math.hu-berlin.de}}
\thanks{}

\author[A. Verra]{Alessandro Verra}
\address{Universit\`a Roma Tre, Dipartimento di Matematica, Largo San Leonardo Murialdo \hfill
 \newline \indent 1-00146 Roma, Italy}
 \email{{\tt
verra@mat.uniroma3.it}}

\begin{abstract}
Using the connection discovered by Hassett between the Noether-Lefschetz moduli space $\cC_{42}$ of special cubic fourfolds of discriminant $42$ and the moduli space $\F_{22}$ of polarized $K3$ surfaces of genus $22$, we show that the universal $K3$ surface over $\F_{22}$ is unirational.
\end{abstract}

\maketitle

\section{Introduction}

The $19$-dimensional moduli space $\F_g$ of polarized $K3$ surfaces of genus $g$ (or of degree $2g-2$), parametrizing pairs  $[S,H]$, where $S$ is a $K3$ surface and $H\in \mbox{Pic}(S)$ is a primitive polarization class satisfying $H^2=2g-2$, is one of the most intriguing parameter spaces in algebraic geometry. In stark contrast to the moduli space of curves or abelian varieties, its Picard group is highly intricate, see \cite{BLMM}. The moduli space $\F_g$ is a quotient of a locally symmetric domain. Via this realization as an orthogonal modular variety one can employ automorphic
methods in order to study its Kodaira dimension. In this way, Gritsenko, Hulek and Sankaran \cite{GHS} proved that $\F_g$ is a variety of general type for $g>62$, as well as for $g=47,51,53,55,58,59,61$. On the other hand, using vector bundles on various rational homogeneous varieties,
in a celebrated series of papers Mukai \cite{M1}, \cite{M2}, \cite{M3}, \cite{M4}, \cite{M5} described the construction of general polarized $K3$ surfaces of genus $g\leq 12$, as well as for $g=13,16,18,20$.  In particular, the moduli space $\F_g$ is unirational for those values of $g$.  The case $g=14$, not covered by Mukai's work, has been settled using the birational isomorphism between $\F_{14}$ and the moduli space $\cC_{26}$ of special cubic fourfolds of discriminant $26$. Nuer \cite{Nu} first showed that $\F_{14}$ is uniruled. This was then improved in \cite{FV}, where we showed that the universal $K3$ surface $\F_{14,1}$ is rational, hence $\F_{14}$ is unirational. Recently Ma \cite{Ma} undertook a systematic study of the Kodaira dimension of the moduli space $\F_{g,n}$ of $n$-pointed $K3$ surfaces of genus $g$, in the spirit of a similar analysis of the Kodaira dimension of $\mm_{g,n}$ carried out in \cite{Log} and \cite{F}.

\vskip 3pt

The aim of this paper is to study the geometry of $\F_{22}$ using the connection between $K3$ surfaces and special cubic fourfolds of discriminant $42$. We establish the following result:

\begin{theorem}\label{main}
The universal $K3$ surface $\F_{22,1}$ of genus $22$ is unirational.
\end{theorem}

In particular, $\F_{22}$ is unirational as well. Note that $22$ is the highest genus where it is known that the moduli space $\F_g$ is \emph{not} of general type. Our approach to $\F_{22}$ relies on the relation between Noether-Lefschetz special cubic fourfolds and polarized $K3$ surfaces, which we explain next.

\vskip 4pt

We fix a smooth cubic fourfold $X\subseteq \PP^5$. Recall the important fact that the Fano variety of lines
$F(X):=\bigl\{\ell\in \GG(1,5):\ell\subseteq X\bigr\}$ is a hyperk\"ahler variety of dimension $4$, see  \cite{BD}. Its primitive cohomology $H^4_{\mathrm{prim}}(X,\mathbb Z)$, displaying the Hodge numbers $(0,\  1,\   20,\  1, \ 0)$, looks like the Tate twist of the middle cohomology of a $K3$ surface, except it has signature $(20,2)$ rather than $(19,3)$. When  $X$ is very general, the lattice $A(X):=H^{2,2}(X)\cap H^4(X, \mathbb Z)$ consists only of classes of complete intersection surfaces, that is,
$A(X)=\langle h^2 \rangle$, where $h\in \mbox{Pic}(X)$ is the hyperplane class, see \cite{V}. Let $\cC$ be the $20$-dimensional coarse moduli space of smooth cubic fourfolds $X\subseteq \PP^5$ and denote by  $\cC_d$ the locus of \emph{special} cubic fourfolds $X$ characterized by the existence of an embedding of a saturated rank $2$ lattice
$$\mathbb L:=\langle h^2, T\rangle \hookrightarrow A(X),$$
of discriminant $\mbox{disc}(\mathbb L)=d$, where $T$ is a codimension $2$ algebraic cycle of $X$ not homologous to a complete intersection. Hassett \cite{H} showed that $\cC_d\subseteq \cC$ is nonempty and if so, an irreducible divisor, if and only if $d>6$ and $d\equiv 0,2 \ (\mbox{mod } 6)$. A conjecture of Kuznetsov \cite{Kuz} predicts that all cubic fourfolds $[X]\in \cC_{2(n^2+n+1)}$ are rational. This has  been confirmed in the classical case $d=14$, see \cite{Fa}, \cite{BR}, and more recently when $d=26$ by Russo and Staglian\`o \cite{RS}. Very recently, the same authors announced a proof of the rationality of all cubic fourfolds from $\cC_{42}$, see \cite{RS2}.

\vskip 4pt

For $d=42$, Hassett's work \cite{H} implies the existence of a rational map of degree $2$
$$
\varphi\colon \F_{22}\rightarrow \cC_{42}, \ \ \varphi([S, H])=[X],
$$
where the cubic fourfold $X$ is characterized by the existence of an isomorphism
\begin{equation}\label{fanolines}
S^{[2]}\cong F(X)\subseteq \GG(1, 5).
\end{equation}

\vskip 4pt

Lai's paper \cite{L} represents an important first step in understanding the relation between $\F_{22}$ and $\cC_{42}$. We summarize its results. Starting with a polarized $K3$ surface $[S,H]\in \F_{22}$, for each point $p\in S$ one considers the rational curve
$$\Delta_p:=\bigl\{\xi \in S^{[2]}: \mathrm{supp}(\xi)=\{p\}\bigr\}.$$

Under the isomorphism $S^{[2]}\cong F(X)$ described above, $\Delta_p$ corresponds
to a rational curve of degree $9$ inside $F(X)\subseteq \GG(1,5)$, that is, to a degree $9$ scroll $R_p\subseteq X$. The double point formula implies that, as long as it has isolated nodal singularities, $R_p$ has $8$ nodes and no further singularities. This is precisely the content of \cite[Theorem 0.3]{L}. We denote by $\mathfrak{H}_{\mathrm{scr}}$ the $PGL(6)$-quotient of the Hilbert scheme of $8$-nodal scrolls $R\subseteq \PP^5$ of degree $9$. Lai shows \cite[Proposition 0.4]{L} that $\mathfrak{H}_{\mathrm{scr}}$ has the expected codimension $8$ inside the parameter space of all scrolls of degree $9$ in $\PP^5$, in particular $\mbox{dim}(\mathfrak{H}_{\mathrm{scr}})=16$.

\vskip 4pt

One can then set up the incidence correspondence between scrolls and cubic fourfolds:
$$\xymatrix{
  & \mathfrak{X}:=\Bigl\{(X,R):R\subseteq X, \mbox{ deg}(R)=9, \ [X]\in \cC_{42}\Bigr\}\big / PGL(6) \ar[dl]_{\pi_1} \ar[dr]^{\pi_2} & \\
   \cC_{42} & & \mathfrak{H}_{\mathrm{scr}}       \\
                 }$$
For a general $[R]\in \mathfrak{H}_{\mathrm{scr}}$ one computes that $h^0\bigl(\PP^5, \mathcal{I}_{R/\PP^5}(2)\bigr)=0$ and $h^0\bigl(\PP^5, \mathcal{I}_{R/\PP^5}(3)\bigr)=6$, It follows that  $\mathfrak{X}$ is birational to a $\PP^{5}$-bundle over the variety $\mathfrak{H}_{\mathrm{scr}}$. Since $\pi_1$ is dominant, this implies that $\cC_{42}$ is uniruled.

\vskip 4pt

This is the point where Lai's paper \cite{L}  ends  and our analysis starts. We first introduce the universal $K3$ surface $u\colon \F_{22,1}\rightarrow \F_{22}$, then the map
$$\tilde{\varphi}\colon \F_{22,1}\rightarrow \mathfrak{X},  \ \ \mbox{ } \ \tilde{\varphi}([S,p]):=[X, R_p],$$
where $R_p$ is the degree $9$ scroll contained in $X$  corresponding to the rational curve $\Delta_p\subseteq F(X)$ under the isomorphism (\ref{fanolines}). We observe that although $\varphi$ has degree $2$, that is, for a general fourfold $[X]\in \cC_{42}$ one has \emph{two} polarized $K3$ surfaces realizing the isomorphism (\ref{fanolines}), this ambiguity disappears once we lift to the universal $K3$ surface. We prove the following:

\begin{theorem}\label{isom}
The map $\tilde{\varphi}\colon \F_{22,1}\rightarrow \mathfrak{X}$ is a birational isomorphism.
\end{theorem}

Since $\mathfrak{X}$ is a $\PP^5$-bundle over $\mathfrak{H}_{\mathrm{scr}}$, the unirationality of $\F_{22,1}$ will be implied by that
of the moduli space $\mathfrak{H}_{\mathrm{scr}}$. To summarize the situation, we have the following commutative diagram:

$$
\xymatrix{
\F_{22,1}  \ar[r]^-{\tilde{\varphi}} \ar[d]^{u} &\mathfrak{X}
 \ar[d]^{\pi_1}  \\
 \F_{22}  \ar[r]^-{\varphi} & \cC_{42}
}$$

We now explain our parametrization of the moduli space of $8$-nodal nonic scrolls.   We start by considering the Hirzebruch surface $\FF_1:=\mbox{Bl}_{\mathpzc{o}}(\PP^2)$, where $\mathpzc{o}\in \PP^2$, and denote by $h$ the class of a line and by $E$ the exceptional divisor. The smooth degree $9$ scroll $R':=S_{4,5}\subseteq \PP^{10}$ is the image of the linear system
$$\phi_{|5h-4E|}\colon \FF_1\hookrightarrow \PP^{10}.$$
We choose a  $4$-plane $\Lambda \in \GG(4, 10)$ which is $8$-secant to the secant variety $\mbox{Sec}(R')\subseteq \PP^{10}$. We may assume $\Lambda\cap R'=\emptyset$ and refer to \cite[Sections 3.2-3.3]{L} for the proof that such a $4$-plane $\Lambda$ exists. Consider the restriction to $R'$ of the projection $\pi_{\Lambda}$ with center $\Lambda$
$$\pi:=\pi_{\Lambda|R'}\colon R'\rightarrow R\subseteq \PP^5,$$  where  $R:=\pi(R')$. Then $R$  is an $8$-nodal scroll of degree $9$. If for $i=1, \ldots, 8$, we have that $\langle x_i, y_i\rangle \cap \Lambda\neq \emptyset$ for certain points $x_i, y_i\in R'$, then the (nodal) singularities of $R$ appear as $n_i:=\pi_{\Lambda}(x_i)=\pi_{\Lambda}(y_i)$. Up to the action of $PGL(6)$ on the ambient projective space $\PP^5$, each $8$-nodal nonic scroll $[R]\in \mathfrak{H}_{\mathrm{scr}}$ appears in this way.

\vskip 4pt

We now fix an unordered set of \emph{four} general rulings $\ell_1, \ell_2, \ell_3, \ell_4$ of $R$, thus they can be assumed to be disjoint from $\mbox{Sing}(R)$. Since containing a line imposes \emph{three} conditions on the linear system of quadrics in $\PP^5$ and since $\mbox{dim } |\OO_{\PP^5}(2)|=20$, it follows that there exists a \emph{unique} quadric $Q\subseteq \PP^5$ containing the rulings $\ell_1, \ldots, \ell_4$, as well as the nodes $n_1, \ldots, n_8$.
We write
\begin{equation}\label{def:gamma}
R\cdot Q=\ell_1+\ell_2+\ell_3+\ell_4+\Gamma.
\end{equation}

It will turn out that  the residual curve  $\Gamma\subseteq \PP^5$  is a degree $14$ integral  curve of arithmetic genus $12$ having nodes at the points $n_1, \ldots, n_8$. Assuming this, let $$C:=\pi^{-1}(\Gamma)\subseteq R'$$ be the normalization of $\Gamma$. Then from (\ref{def:gamma}) we find that $C\in |6h-4E|$. Therefore  $C$ is a hyperelliptic curve of genus $4$ which passes through the points $x_i, y_i\in R'$, for $i=1, \ldots, 8$.  The degree $2$ pencil on $C$ is cut out by the rulings of $R'$, that is, $\OO_C(h-E)\in W^1_2(C)$. Denoting by $\iota:C\rightarrow C$ the hyperelliptic
involution, we observe that
$$R=\bigcup_{x\in C} \bigl\langle \pi(x), \pi(\iota(x))\bigr\rangle\subseteq \PP^5,$$
that is, the degree $9$ scroll $R$ can be recovered from the curve $\Gamma\subseteq \PP^5$.

\vskip 4pt

We denote by $\P$ the parameter space of pairs $[R, \ell_1+\cdots+\ell_4]$, where $R\subseteq \PP^5$ is an $8$-nodal scroll of degree $9$ and $\ell_1, \ldots, \ell_4$ are rulings of $R$, viewed as an \emph{unordered} set.  In the definition of $\P$ we quotient out by the $PGL(6)$-action on $\PP^5$. The forgetful map $$\P \rightarrow \mathfrak{H}_{\mathrm{scr}}$$ is birational to a $\PP^1$-bundle corresponding to the choice of the four rulings, in particular $\mbox{dim}(\P)=17$. Let $\mathfrak{Hyp}_{4,8}$ be the moduli space of pairs $[\Gamma, L]$, where $\Gamma$ is an integral $8$-nodal curve of arithmetic genus $12$, whose normalization $\nu:C\rightarrow \Gamma$ is a hyperelliptic curve of genus $4$ and $L\in W^2_8(\Gamma)$, that is, $L$  is a line bundle of degree $8$ on $\Gamma$ with $h^0(\Gamma, L)\geq 3$.  Note that by Riemann-Roch, in this case $\omega_{\Gamma}\otimes L^{\vee}\in W^5_{14}(\Gamma)$. We have the following result, reducing the study of $\F_{22,1}$ to that of a certain moduli space of curves.

\begin{theorem}\label{unir}
There exists a birational isomorphism $\chi\colon \P \stackrel{\cong}\dashrightarrow \mathfrak{Hyp}_{4,8}$ \ given by
$$\chi\bigl([R, \ell_1+\ell_2+\ell_3+\ell_4]\bigr)=[\Gamma, \omega_{\Gamma}(-1)].$$
\end{theorem}

Theorem \ref{main} now follows once we establish the unirationality of $\mathfrak{Hyp}_{4,8}$. We indicate how to carry this out. Start with a general element $[\Gamma, L]\in \mathfrak{Hyp}_{4,8}$, viewed as an $8$-nodal degree $14$ curve $\Gamma\subseteq \PP^5$ embedded by the line bundle $\omega_{\Gamma}\otimes L^{\vee}$. We shall show that a suitably general such curve $\Gamma$ is projectively normal, thus the kernel of the multiplication map
$$\mbox{Sym}^2 H^0(\Gamma, \OO_{\Gamma}(1))\rightarrow H^0(\Gamma, \OO_{\Gamma}(2))$$
is $4$-dimensional. We can write
$$\mathrm{Bs} \ \bigl|\mathcal{I}_{\Gamma/\PP^5}(2)\bigr|=\Gamma+B.$$
The residual curve  $B\subseteq \PP^5$ is a conic such that $\Gamma \cdot B=6$. We denote by $\Pi:=\langle B\rangle \subseteq \PP^5$ the plane spanned by $B$. There exists a $3$-dimensional subspace $V\subseteq H^0\bigl(\PP^5, \mathcal{I}_{\Gamma/\PP^5}(2)\bigr)$ consisting of quadrics containing the plane $\Pi$. We write
\begin{equation}\label{defi:T}
\mathrm{Bs}\ |V|=\Pi+T,
\end{equation}
where $T\subseteq \PP^5$ is a degree $7$ surface lying on three quadrics that intersect along the  $2$-plane $\Pi$. It is not hard to see
that $T\cong \mathrm{Bl}_9(\PP^2)$ is the blow-up of $\PP^2$ at $9$ general points in $\PP^2$. Moreover, the map
$\varphi: \mathrm{Bl}_9(\PP^2)\hookrightarrow T\subseteq \PP^5$ implicitly defined by (\ref{defi:T}) is induced by the linear system
$$|4h-E_1-\cdots-E_9|,$$
where $E_1, \ldots, E_9$ are the exceptional divisors. Via the isomorphism $T\cong \mathrm{Bl}_9(\PP^2)$, one realizes $\Gamma$ as an octic plane curve with $17$ nodes divided in two groups: namely the $9$ points where $\PP^2$ is blown up and the remaining $8$ nodes. This plane model is  helpful to prove the next result:

\begin{theorem}\label{hyp_uni}
The moduli space $\mathfrak{Hyp}_{4,8}$ is unirational.
\end{theorem}

To prove Theorem \ref{hyp_uni}, we fix a cubic scroll $Z\subseteq \PP^4$ obtained by embedding the Hirzebruch surface
$\FF_1:=\mbox{Bl}_{\mathpzc{o}}(\PP^2)$ by the linear system $|2h-E|$. We consider the parameter space

\begin{align*}
\mathcal{T}=\Bigl\{(t_1, \ldots, t_8, \ell, C): t_i\in Z \mbox{ for } i=1, \ldots, 8, \ \ell\in \GG(1,4) \mbox{ is a  line in } \PP^4, \\
C\in \bigl|\mathcal{I}_{\{x_1, y_1, \ldots, x_8, y_8\}/Z}(6h-4E)\bigr|, \mbox{ where } \langle \ell, t_i\rangle\cdot Z =t_i+x_i+y_i, \mbox{ for } i=1, \ldots,8.\Bigr\}
\end{align*}

Note that $\mbox{dim } \bigl|\mathcal{I}_{\{x_1, y_1, \ldots, x_8, y_8\}/Z}(6h-4E)\bigr|=1$, hence the  map $\mathcal{T}\rightarrow Z^8\times \GG(1,4)$ is birationally a locally trivial $\PP^1$-bundle over a rational variety. Therefore $\mathcal{T}$ is rational. The curve $C$ is hyperelliptic of genus $4$. Denoting by $\pi_{\ell}:\PP^4\dashrightarrow \PP^2$ the projection with center $\ell$, observe that $n_i':=\pi_{\ell}(x_i)=\pi_{\ell}(y_i)$ for $i=1, \ldots, 8$.  The dominant rational map
$$\vartheta\colon \mathcal{T}\dashrightarrow \mathfrak{Hyp}_{4,8}$$
needed to prove Theorem \ref{hyp_uni} is obtained by associating to the point $\bigl(t_1, \ldots, t_8, \ell, C\bigr)\in \mathcal{T}$ essentially the projected curve $\Gamma:=\pi_{\ell}(C)$. This is a nodal octic plane curve having $8$ distinguished nodes at $n_1', \ldots, n_8'$, as well as $9$ further nodes. The image under the map $\varphi$ of the proper transform of $\Gamma'$ in the blow-up of $\PP^2$ at these $9$ points gives rise to an element of $\mathfrak{Hyp}_{4,8}$. For further details on the definition of the map $\vartheta$ we refer to Theorem \ref{thm:domin}.

\vskip 3pt

It turns out that proving directly the various transversality assumptions implicit in this sketched proof of Theorem \ref{unir} is not straightforward. Instead, in the rest of the paper we shall reverse the argument presented in the Introduction. First we show that $\mathfrak{Hyp}_{4,8}$ is unirational (see Theorem \ref{thm:domin}), then using the explicit unirational parametrization found in this way, we show that the map
$\chi\colon \mathcal{P}\dashrightarrow \mathfrak{Hyp}_{4,8}$ is well defined, as well as birational.

\vskip 4pt

\noindent {\bf Acknowledgment:} We are grateful to the referee for a very careful reading of the paper and for many good suggestions that improved
the presentation. 

\section{The moduli space $\F_{22}$ via special cubic fourfolds}

We denote by $\mathcal F_g$  the irreducible $19$-dimensional moduli space of smooth polarized $K3$ surfaces $[S,H]$  of genus $g$, that is, with $H\in \mbox{Pic}(S)$ being a nef class satisfying $H^2=2g-2$. Let $u\colon \F_{g,1}\rightarrow \F_g$ be  the universal $K3$ surface of genus $g$ in the sense of stacks. Each fiber $u^{-1}([S,H])$ is thus identified with the $K3$ surface $S$.

\vskip 4pt

We fix a smooth cubic fourfold $X\subseteq \PP^5$ and denote by $h$ its hyperplane class. The Hodge structure on the primitive cohomology
$H^4_{\mathrm{prim}}(X,\mathbb Z)$ is similar to the twist of the middle cohomology of a $K3$ surface. Since the signatures (with respect to the intersection form) are different, (20, 2) and (19, 3) respectively, one has to pass to sub-Hodge structures of codimension one, both having signature $(19,2)$, to have the possibility of realizing an isomorphism of Hodge structures between the two sides. On the cubic fourfold side one requires the existence of a class $T\in H^{2,2}(X)$, whereas on the $K3$ side one requires the existence of a polarization $H\in \mbox{Pic}(S)$ such that the following isomorphism of Hodge structures holds
\begin{equation}\label{hassettisom1}
\langle h^2, T\rangle ^{\perp}\cong H^2_{\mathrm{prim}}(S,\mathbb Z)(-1).
\end{equation}
Denoting by $d:=\mbox{disc}(\langle h^2, T\rangle)=H^2$, it is proved in \cite[Theorem 5.1.3]{H} that the isomorphism (\ref{hassettisom1}) is realized for any $d>6$ such that $d\equiv 0, 2 \ (\mbox{mod } 6)$ that is not divisible by $4, 9$ or by any prime $p\equiv 2 \ (\mbox{mod } 3)$.
When $d=2(n^2+n+1)$, the isomorphism (\ref{hassettisom1}) takes the geometric form (\ref{fanolines})
$$S^{[2]}\cong F(X)\subseteq \GG(1,5).$$
This opens the way to a study of the moduli spaces $\F_{n^2+n+2}$ where $n\geq 2$, using the concrete projective geometry of cubic fourfolds. The case $n=2$ (that is, $d=14$) is classical and essentially due to Fano \cite{F}; we refer to \cite{BD} and \cite{BR} for a modern perspective and stronger results. The case $n=3$ (that is, $d=26$) has been treated in our paper
\cite{FV} as well as in \cite{RS}, whereas this paper is devoted to the case $n=4$ (that is, $d=42$).

\vskip 3pt

For $d=42$, Hassett \cite{H}
constructed a degree $2$ map
$$\varphi\colon \F_{22}\longrightarrow \cC_{42}, \ \varphi\bigl([S,H]\bigr)=[X],$$
such that the isomorphism (\ref{fanolines}) holds. Note that $\varphi$ is defined at the level of moduli spaces of weight $2$ Hodge structures and there is no direct geometric construction of the cubic fourfold one associates to a $K3$ surface of genus $22$. Since $\mbox{deg}(\varphi)=2$, it follows that for a general $[X]\in \cC_{42}$ there exist two distinct polarized $K3$ surfaces $[S,H]$ and $[S',H']$ such that
$$S^{[2]}\cong S'^{[2]}\cong F(X).$$
Clarifying the relation between $S$ and $S'$ is essential in order to prove Theorem \ref{isom}.

\vskip 4pt

\subsection{Hilbert squares of $K3$ surfaces.}
Let $(S,H)$ be a $K3$ surface with $\mbox{Pic}(S)=\mathbb Z\cdot H$ and $H^2=2g-2$. We denote by $S^{[2]}$ the Hilbert scheme of length two zero-dimensional subschemes on $S$.
Then $H^2(S^{[2]}, \mathbb Z)$ is endowed with the \emph{Beauville-Bogomolov} quadratic form $q$. Let $\Delta\subseteq S^{[2]}$ be the divisor consisting of
zero-dimensional subschemes  supported only at a single point and denote by $\delta:=\frac{[\Delta]}{2}\in H^2(S^{[2]},\mathbb Z)$ the reduced diagonal class. Then $q(\delta, \delta)=-2$. Moreover
$$\Delta=\PP(T_S)=\bigcup_{p\in S} \Delta_p,$$  where $\Delta_p$ is the rational curve consisting of those $0$-dimensional subschemes $\xi\in \Delta$ such that
$\mbox{supp}(\xi)=\{p\}$. We set $\delta_p:=[\Delta_p]\in H_2(S^{[2]}, \mathbb Z)$.

\vskip 4pt

For any curve $C\in |H|$, we introduce the divisor
$$f_C:=\bigl\{\xi\in S^{[2]}:\mbox{supp}(\xi)\cap C\neq \emptyset\bigr\}$$
and set $f:=[f_C]\in H^2\bigl(S^{[2]},\mathbb Z\bigr)$.
For a point $p\in S$, we also define the curve
$$F_p:=\bigl\{\xi=p+x\in S^{[2]}:x\in C\bigr\}$$
and set $f_p:=[F_p]\in H_2(S^{[2]},\mathbb Z)$.
The Beauville-Bogomolov form can be extended to a quadratic form on  $H_2(S^{[2]}, \mathbb Z)$, by
setting $q(\alpha, \alpha):=q(w_{\alpha}, w_{\alpha})$, with $w_{\alpha}\in H^2(S^{[2]},\mathbb Z)$ being the class characterized by
the property $\alpha\cdot u=q(w_{\alpha},u)$, for every $u\in H^2(S^{[2]},\mathbb Z)$. Here $\alpha \cdot u$ denotes the intersection product.

\vskip 4pt

One has the following decompositions, orthogonal with respect to $q$, both for the Picard group and for the group $N_1(S^{[2]},\mathbb Z)$ of $1$-cycles modulo numerical equivalence:
$$\mbox{Pic}(S^{[2]})\cong \mathbb Z\cdot f\oplus \mathbb Z\cdot \delta \  \mbox{ and } \ N_1(S^{[2]},\mathbb Z)\cong \mathbb Z\cdot f_p \oplus \mathbb Z\cdot  \delta_p.$$
We record the following more or less immediate relations:
\begin{equation}\label{intprod}
f\cdot f_p=2g-2, \ \delta\cdot \delta_p=-1, \ f\cdot \delta_p=0 \mbox{ and } \delta\cdot f_p=0.
\end{equation}
The form $q$ takes the following values on $H_2\bigl(S^{[2]},\mathbb Z\bigr)$:
$$q(f_p, f_p)=2g-2, \ \ q(f_p, \delta_p)=0, \ \ q(\delta_p, \delta_p)=-\frac{1}{2}.$$
Thus $q(af_p-b\delta_p)=a^2(2g-2)-\frac{b^2}{2}$, for $a,b\in \mathbb Z$.

\vskip 4pt

It follows from \cite[Proposition 13.1]{BM} (see also \cite[Proposition 3.14]{DM} for this formulation) that for a polarized $K3$ surface $[S,H]\in \F_{22}$ with $\mbox{Pic}(S)=\mathbb Z\cdot H$, the nef cone
$\mbox{Nef}(S^{[2]})$ equals the movable cone $\mbox{Mov}(S^{[2]})$ and it is generated by the rays $f$ and $55f-252\delta$ respectively. Using the terminology of \cite{H}, the Hilbert square $S^{[2]}$ is \emph{strongly ambiguous}, that is, there exists another $K3$ surface $S'$ such that there exists an isomorphism $r\colon S^{[2]}\stackrel{\cong}\rightarrow S'^{[2]}$ which is \emph{not induced} by an automorphism $S\stackrel{\cong}
\rightarrow S'$. This implies $r^*(\delta')\neq \delta$ and then necessarily, the map $r^*\colon H^2\bigl(S'^{[2]}, \mathbb Z\bigr)\rightarrow H^2\bigl(S^{[2]}, \mathbb Z\bigr)$ interchanges the two rays of the respective nef cones, that is,
$$r^*(f')=55f-252\delta, \ \mbox{ } \ \ r^*(55f'-252\delta')=f.$$
Then also $r_*(f)=55f'-252\delta'$ and $r_*(\delta)=12f'-55\delta'$, from which we obtain the following relations at the level of the cone of curves in $S^{[2]}$ and $S'^{[2]}$ respectively:

\begin{equation}\label{eqcurv}
r^*(\delta_{p'}')=6f_p-55\delta_p, \ \ \mbox{  } \ r^*(6f_{p'}'-55\delta_{p'}')=\delta_p.
\end{equation}

\vskip 3pt

\subsection{Scrolls contained in special cubic fourfolds.}
Suppose $R\subseteq X\subseteq \PP^5$ is a rational scroll with smooth normalization having only isolated singularities and which is contained in a cubic fourfold $X$. The \emph{double point formula} \cite[Theorem 9.3]{Ful} gives the number $D(R)$ of singularities of $R$, counted appropriately:
\begin{equation}\label{doublepoints}
2D(R)=R^2-6h^2-K_R^2-3h\cdot K_R+\chi_{\mathrm{top}}(R).
\end{equation}
If moreover all singularities of $R$ are nodal, then $D(R)$ equals the number of nodes of $R$.

\vskip 3pt

When $[X]\in \cC_{42}$, assuming that $A(X)=\langle h^2,[R]\rangle$, where $h^2\cdot [R]=\mbox{deg}(R)=9$, necessarily $R^2=41$. Since $\chi_{\mathrm{top}}(R)=4$ and $h\cdot K_R=-11$, from
(\ref{doublepoints}), we compute $D(R)=8$. Therefore if $R$ has only isolated nodes, then it is necessarily $8$-nodal.

\begin{proposition}\label{scrolls1}
Suppose $[S,H]\in \F_{22}$ is an element such that $\mathrm{Pic}(S)=\mathbb Z\cdot H$ and let $Z\subseteq S^{[2]}$ be an effective $1$-cycle of degree $9$ with respect to the Pl\"ucker embedding. Then $[Z]= f_p$ or $[Z]=6f_p-55\delta_p$. In the first case $Z=\Delta_p$ for some point $p\in S$, and in the second case $r(Z)=\Delta_{p'}$ for some point $p'\in S$, where $r:S^{[2]}\stackrel{\cong}\rightarrow S'^{[2]}$.
\end{proposition}
\begin{proof}
Assume that $Z$ is an effective $1$-cycle on $S^{[2]}$ and write $[Z]=af_p-b\delta_p\in N_1\bigl(S^{[2]},\mathbb Z\bigr)$. Let $\gamma_S$ denote the class of the Pl\"ucker line bundle $\OO_{S^{[2]}}(1)$ with respect to the isomorphism $S^{[2]}\cong F(X)$. Since $q(\gamma_S, \gamma_S)=6$, one obtains
$$\gamma_S=2f-9\delta\in H^2\bigl(S^{[2]},\mathbb Z\bigr).$$
Therefore $9=Z\cdot \gamma_S=
(af_p-bf_p)(2f-9\delta)=84a-9b$, hence we can write $a=3a_1$, with $a_1\in \mathbb Z$, in which case $b=28a_1-1$. Using \cite[Proposition 12.6]{BM}, if $Z$ is effective we also have the inequality $q(Z,Z)\geq -\frac{5}{2}$, implying $7a_1^2-14a_1-1\leq 0$.

\vskip 3pt

The integer solutions of this inequality are $a_1=0$, when $[Z]=\delta_p$, $a_1=2$, in which case $[Z]=6f_p-55\delta_p$, and finally $a_1=1$. Note that in the first two cases $q(Z,Z)=-\frac{1}{2}$. On the other hand, $a_1=1$ implies $[Z]=3f_p-27\delta_p$, yielding
$q(Z,Z)=\frac{27}{2}$. But this is incompatible with the double point formula. Indeed, if $R\subseteq X$ is the scroll associated to the curve $Z$ under the isomorphism $S^{[2]}\cong F(X)$, then following \cite[§ 7.1]{HT}, we have
$$R^2=\frac{(Z\cdot \gamma_S)^2}{2}-q(Z,Z),$$
which is impossible because $R^2=41$.

\vskip 3pt

In the case $[Z]=6f_p-55\delta_p$, denoting by $[S', H']\in \F_{22}$ the polarized $K3$ surface such that
$r\colon S^{[2]}\stackrel{\cong}\rightarrow S'^{[2]}\cong F(X)$, it follows from (\ref{eqcurv}) that $[r_*(Z)]=\delta_{p'}'$.
By possibly replacing $S$ with $S'$, we may assume $[Z]=\delta_p$. We claim this implies $Z$ is one of the smooth rational curves
$\Delta_p$, for some point $p\in S$.

\vskip 3pt

Indeed, from $[Z]\cdot \delta=-1$, it follows that $Z\subseteq \Delta$. Moreover,  $Z$ lies in one of the fibers of the $\PP^1$-bundle $\pi\colon \Delta=\PP(T_S)\rightarrow S$. This implies that $Z=\Delta_p$, for some $p\in S$, because  otherwise
$\pi(Z)\equiv mH$ for some integer $m>0$ and then
$$mH^2=Z\cdot \pi^{-1}(H)=Z\cdot f=\delta_p\cdot f=0,$$
which is a contradiction.
\end{proof}

\vskip 3pt

We are now in a position to prove Theorem \ref{isom}. Recall the definition given in the Introduction of the parameter space $\mathfrak{X}$
of pairs $(X,R)$, where $[X]\in \cC_{42}$ and $R\subseteq X$ is a degree $9$ scroll. As explained, as long as $R$ has isolated nodal singularities, $R$ has precisely $8$ nodes. We define the map $ \tilde{\varphi}\colon \F_{22,1}\rightarrow \mathfrak{X}$ given by $\tilde{\varphi}([S,H,p]):=[X, R_p]$, where the cubic scroll $X$ is determined by the isomorphism $F(X)\cong S^{[2]}$ and the scroll $R_p\subseteq X$ corresponds to $\Delta_p$, viewed as a rational curve inside $F(X)$. Recall that it is proved in \cite{L} that the projection $\pi_1\colon \mathfrak{X}\rightarrow \cC_{42}$ is dominant. This implies that $\tilde{\varphi}$ is well-defined.

\vskip 5pt

\noindent {\emph{Proof of Theorem \ref{isom}.} We show that $\tilde{\varphi}\colon \F_{22,1}\rightarrow \mathfrak{X}$ is a birational isomorphism by constructing its inverse. Start with an element $[X,R]\in \mathfrak{X}$ and we denote by $\bigl\{[S,H],[S',H']\bigr\}=\varphi^{-1}([X])$ the two polarized $K3$ surfaces realizing the isomorphism (\ref{fanolines}).  Applying Proposition \ref{scrolls1}, for precisely one element of $\varphi^{-1}([X])$, say $[S,H]$ we have that the curve $Z=Z_R\subseteq F(X)\cong S^{[2]}$ of rulings of $R$ has class $[Z]=\delta_p$, for a point $p\in S$. Then clearly $\tilde{\varphi}^{-1}([X,R])=[S,H,p]$.
\hfill $\Box$

\section{Moduli of nodal hyperelliptic curves}

On our way towards establishing Theorems \ref{unir} and \ref{hyp_uni} and ultimately proving Theorem \ref{main}, we shall reverse the construction described in the Introduction associating to a suitably general scroll $[R]\in \hh$ an $8$-nodal curve with hyperelliptic  normalization. In order to establish the various transversality claims mentioned in the Introduction, we find it easier to start with a suitable nodal hyperelliptic curve and bring the degree $9$ scroll into picture only later. We begin therefore by introducing and studying various moduli spaces of curves that will turn out to be relevant when dealing with $\cC_{42}$.

\vskip 2pt

Recall that for an irreducible nodal curve $Y$, we denote by $W^r_d(Y)$ the Brill-Noether locus consisting of line bundles $L\in \mbox{Pic}^d(Y)$ satisfying $h^0(Y, L)\geq r+1$.

\begin{definition}
We denote by $\mathfrak{Hyp}_{4,8}$ the moduli space of pairs $[\Gamma, L]$, where $\Gamma$ is an irreducible $8$-nodal curve of arithmetic genus $12$, such that its normalization $C\rightarrow \Gamma$ is hyperelliptic and $L\in W^2_8(\Gamma)$.
\end{definition}

In this Section we provide an explicit parametrization of $\mathfrak{Hyp}_{4,8}$ and conclude that this space is unirational. We begin with some preparation. We consider the Hirzebruch surface $\FF_1:=\mathrm{Bl}_{\mathpzc{o}}(\PP^2)$ viewed  as a cubic scroll via the embedding
\begin{equation}\label{cubic_scroll}
\phi_{|2h-E|}\colon \FF_1\hookrightarrow Z\subseteq \PP^4.
\end{equation}
Here $h$ denotes the pull-back of the line class under the contraction morphism $\FF_1\rightarrow \PP^2$, whereas $E$ is the exceptional divisor over the point $\mathpzc{o}\in \PP^2$.

\vskip 3pt

We denote by $\mathcal{H}_4$ the moduli space of smooth hyperelliptic curves of genus $4$ and by $\mathcal{P}ic^8_{\mathcal{H}_4}\rightarrow \H_4$ the universal Picard variety of pairs $[C,L]$, where $[C]\in \mathcal{H}_4$ and $L\in \mbox{Pic}^8(C)$. Our next result provides an explicit birational realization of this universal Picard variety.

\begin{proposition}\label{pic:hyp}
There is a birational isomorphism $\mathcal{P}ic^8_{\H_4}\stackrel{\cong}\dashrightarrow \bigl|6h-4E\bigr|/\mathrm{Aut}(\mathrm{\textbf{F}}_1)$.
\end{proposition}
\begin{proof}
A smooth curve $C\in |6h-4E|$ is hyperelliptic of genus $4$. To it we can associate the pair $[C,L]$, where $L:=\OO_C(2h-E)$ is a line bundle of degree $8$. In other words, $L=\OO_C(1)$, where $C\subseteq Z\subseteq \PP^4$ is viewed as an octic curve. This construction is obviously $\mbox{Aut}(\FF_1)$-invariant, hence it gives rise to a map $\bigl|6h-4E\bigr|/\mbox{Aut}(\FF_1)\dashrightarrow \mathcal{P}ic^8_{\H_4}$.

\vskip 4pt

Conversely, we start with a general line bundle $L\in \mbox{Pic}^8(C)$ on a hyperelliptic curve $C$ of genus $4$. We denote by $A\in W^1_2(C)$ the hyperelliptic pencil. We may assume $L$ does not lie in the translate  $\omega_C\otimes A+C-C\subseteq \mbox{Pic}^8(C)$ of the difference variety $C-C\subseteq \mbox{Pic}^0(C)$. Set $\OO_C(h):=L\otimes A^{\vee}\in \mbox{Pic}^6(C)$. Then $h^0(C,\OO_C(h))=3$ and our assumption on $L$ guarantees an induced regular map $\phi_{|h|}\colon C\rightarrow \PP^2$, whose image is a sextic curve $C'\subseteq \PP^2$. Set
$$N:=\OO_C(h)\otimes A^{\vee}=L\otimes A^{-2}.$$ For a general $L\in \mbox{Pic}^8(C)$ we have $h^0(C, N)=1$ and we write $N=\OO_C(x_1+x_2+x_3+x_4)$, for points $x_i\in C$. By choosing $L$ generally in $\mbox{Pic}^8(C)$ we can arrange that the points $x_i$ are distinct. Then $$h^0\bigl(C, \OO_C(h)(-x_1-x_2-x_3-x_4)\bigr)=h^0(C, A)=2,$$ which is to say that the image $C':=\phi_{|h|}(C)$ has a $4$-fold point at $\mathpzc{o}:=\phi_{|h|}(x_i)$ for $i=1, \ldots, 4$. Comparing the genera of $C$ and $C'$ we see that $C'$ has no further singularities. This implies we can embed $C$ in the blown-up surface $\mathrm{Bl}_{\mathpzc{o}}(\PP^2)$ such that $C\in |6h-4E|$. Since $A=\OO_C(h-E)$, we also obtain $L=\OO_C(h)\otimes A=\OO_C(2h-E)$, thus finishing the proof.
\end{proof}

In our study of the moduli space $\cC_{42}$ via nodal scrolls, a special role is played by a certain degree $7$ rational surface in $\PP^5$. In what follows, we summarize its properties. If $\mathpzc{o}_1, \ldots, \mathpzc{o}_n\in \PP^2$ are distinct points, we denote by $\mbox{Bl}_n(\PP^2):=\mbox{Bl}_{\mathpzc{o}_1, \ldots, \mathpzc{o}_n}(\PP^2)$ their blow-up, by $E_i$ the exceptional divisor over $\mathpzc{o}_i$, and by $h$ the pull-back of the line class under the contraction morphism $\mbox{Bl}_n(\PP^2)\rightarrow \PP^2$.

\begin{proposition}\label{prop:T}
Let $\mathpzc{o}_1, \ldots, \mathpzc{o}_9\in \PP^2$ be points lying on a \emph{unique smooth} cubic curve. Then the linear system $|4h-E_1-\cdots -E_9|$ is very ample on $\mathrm{Bl}_9(\PP^2)$ and  the image $T$ of the embedding $$\phi_{|4h-E_1-\cdots-E_9|}\colon \mathrm{Bl}_{9}(\PP^2)\hookrightarrow \PP^5$$ is projectively normal. In particular $h^0\bigl(\PP^5, \mathcal{I}_{T/\PP^5}(2)\bigr)=3$. Furthermore, $\mathrm{Bs}\ \bigl|\mathcal{I}_{T/\PP^5}(2)\bigr|=T\cup \Pi$, where $\Pi$ is a $2$-plane meeting $T$ along a smooth elliptic curve.
\end{proposition}
\begin{proof}
The fact that the linear system $|4h-E_1-\cdots-E_9|$ is very ample on the blow-up $\mbox{Bl}_9(\PP^2)$ follows from \cite[Theorem 2.2]{Co}.
Let $C\in |\OO_T(1)|$ be a general hyperplane section on $T$. Then $C$ is a smooth non-hyperelliptic curve of genus $3$. Since $\mbox{deg}(\OO_C(1))=7$, using for instance \cite[page 55]{Mu}, it follows that the curve $\phi_{|\OO_C(1)|}\colon C \hookrightarrow \PP^4$ is projectively normal. Since $h^0(C,\OO_C(2))=12$, we find  $h^0\bigl(\PP^4, \mathcal{I}_{C/\PP^4}(2)\bigr)=3$. Denoting by $\lambda \in H^0(T, \OO_T(1))$ the equation of the hyperplane $\langle C\rangle\subseteq \PP\bigl(H^0(\OO_T(1))^{\vee}\bigr)\cong \PP^5$ spanned by $C$, we have a short  exact sequence
$$0\longrightarrow H^0(T, \OO_T(1))\stackrel{\cdot \lambda}\longrightarrow H^0(T, \OO_T(2))\longrightarrow H^0(C, \OO_C(2))\longrightarrow 0,$$
from which we compute $h^0(T, \OO_T(2))=6+12=18$. We also have the commutative diagram
$$
\xymatrix{
0 \ar[r]  &\lambda \cdot H^0(T, \OO_T(1))  \ar[r] \ar[d]^{\cong} &\mbox{Sym}^2 H^0 (T, \OO_T(1))
\ar[r]  \ar[d]^{\mu_T} &\mbox{Sym}^2 H^0(C, \OO_C(1)) \ar[r] \ar[d]^{\mu_C} &0 \\
0 \ar[r] & H^0(T, \OO_T(1))  \ar[r]^{\cdot \lambda} &H^0(T, \OO_T(2)) \ar[r] &H^0(C,\OO_C(2)) \ar[r] &0
},$$
where $\mu_T$ and $\mu_C$ denote the multiplication maps. It follows that
$\mbox{Coker}(\mu_T)\cong \mbox{Coker}(\mu_C)$, that is, $T$ is projectively normal. In particular
$T\subseteq \PP^5$ lies on precisely three quadrics.

\vskip 4pt

The smooth elliptic curve $J\in |3h-E_1-\cdots-E_9|$ can be viewed as a cubic curve in $\PP^5$ spanning the plane $\Pi$. Any quadric containing $T$ also contains $J$, thus  $\Pi\subseteq \mathrm{Bs }\ \bigl| \mathcal{I}_{T/\PP^5}(2)\bigr|$.
The multiplication map $H^0(T,\OO_T(h))\otimes H^0(T, \OO_T(J))\rightarrow H^0(T,\OO_T(1))$ being obviously injective, we find that $T\cap \Pi=J$.
We finally claim that
\begin{equation}\label{bloc}
\mathrm{Bs } \ \bigl| \mathcal{I}_{T/\PP^5}(2)\bigr|=T\cup \Pi.
\end{equation}
Indeed, one inclusion having been already established, suppose by contradiction there is a point $\mathpzc{r}\in \mathrm{Bs } \ \bigl| \mathcal{I}_{T/\PP^5}(2)\bigr|\setminus (T\cup \Pi)$. We pick a general hyperplane hyperplane $\PP^4\cong H\subseteq \PP^5$ passing through $\mathpzc{r}$. Then $T\cap H=:C$ is a smooth non-hyperelliptic curve of genus $3$, where $\OO_C(1)\in \mbox{Pic}^7(C)$, whereas $H\cap \Pi=:\ell$ is a line. The components $C$ and $\ell$ meet along the divisor $\mathpzc{r}_1+\mathpzc{r}_2+\mathpzc{r_3}$ consisting of the points lying on the intersection $J\cdot H\subseteq \Pi\cap H$. Furthermore,
$$H^0\bigl(C, \OO_C(\mathpzc{r}_1+\mathpzc{r}_2+\mathpzc{r}_3)\bigr)\cong H^0(C, \OO_C(J))\cong H^0(T, \OO_T(J))$$
is a $1$-dimensional space, for the cubic $J$ does not move in its linear system. It follows that the stable genus $5$ curve $C\cup \ell$ is not trigonal. Therefore, its canonical embedding $C\cup \ell\hookrightarrow \PP^4$ is ideal-theoretically cut out by quadrics, in particular $\mathpzc{r}\in \mbox{Bs}\ \bigl|\mathcal{I}_{C\cup \ell/\PP^4}(2)\bigr|=C\cup \ell$. In particular $\mathpzc{r}\in T$, which shows that $T$ is scheme-theoretically cut out by quadrics.
\end{proof}

\vskip 5pt

We describe a  geometric construction that will yield a parametrization of $\mathfrak{Hyp}_{4,8}$. Recall that $Z\subseteq \PP^4$ denotes the cubic scroll defined by (\ref{cubic_scroll}). We fix general points $(t_1, \ldots, t_8)\in Z^8$ and a general line $\ell\subseteq \PP^4$ disjoint from $Z$. For $i=1, \ldots, 8$, we obtain further points $x_i, y_i\in Z$ via the relation
\begin{equation}\label{def:xyt}
\langle \ell, t_i\rangle \cdot Z=t_i+x_i+y_i,
\end{equation}
with the intersection being taken inside $\PP^4$.

\vskip 4pt

We consider the projection $\pi_{\ell}\colon \PP^4 \dashrightarrow \PP^2$ of center $\ell$, whose restriction $\pi_{\ell}\colon Z\rightarrow \PP^2$ is a regular morphism of degree $3$. Since
$\langle \ell, x_i\rangle=\langle \ell, y_i\rangle=\langle \ell, t_i\rangle$, it follows that $$\pi_{\ell}(x_i)=\pi_{\ell}(y_i)=\pi_{\ell}(t_i)\in \PP^2.$$ Furthermore, let us choose a general curve
$C\in \bigl|\mathcal{I}_{\{x_1, y_1, \ldots, x_8, y_8\}/Z}(6h-4E)\bigr|$. Note that
$$\mbox{dim }\bigl|\mathcal{I}_{\{x_1, y_1, \ldots, x_8, y_8\}/Z}(6h-4E)\bigr|={8\choose 2}-{5\choose 2}-1-16=1.$$

\begin{definition}\label{def:T}
Let $\mathcal{T}$ be the space of triples $(t_1, \ldots, t_8, \ \ell, \ C)\in Z^8\times \GG(1,4)\times |6h-4E|$,
where $C\in \bigl|\mathcal{I}_{\{x_1, y_1, \ldots, x_8, y_8\}/Z}(6h-4E)\bigr|$ is a nodal curve, with the points $x_i, y_i\in Z$ described by (\ref{def:xyt}).
\end{definition}

Clearly,  $\mathcal{T}$ is a locally trivial $\PP^1$-bundle over $Z^8\times \GG(1,4)$. In particular, $\mathcal{T}$ is a rational
variety of dimension $23$. Note that the $6$-dimensional automorphism group $\mbox{Aut}(\FF_1)$ acts on $\mathcal{T}$, where the action on $\GG(1,4)$ is via the identification $\mbox{Aut}(\FF_1)\cong \mbox{Aut}(Z)\subseteq PGL(5)$. Therefore the quotient
$\mathcal{T}/\mbox{Aut}(\FF_1)$ has dimension $17$ (and is of course unirational).

\begin{theorem}\label{thm:domin}
One has a dominant rational morphism $\vartheta\colon \mathcal{T} \dashrightarrow \mathfrak{Hyp}_{4,8}$. In particular, $\mathfrak{Hyp}_{4,8}$ is unirational
\end{theorem}

\begin{proof}
We start with a suitably general point $(t_1, \ldots, t_8, \ell, C)\in \mathcal{T}$. In particular $C\subseteq Z\subseteq \PP^4$ is a smooth hyperelliptic curve of genus $4$ and the projection $\pi_{\ell}:Z\rightarrow \PP^2$ is regular.

\vskip 4pt

\noindent {\bf  Claim (i):} The image $\Gamma'$ of the projection $\pi:=\pi_{\ell |C}\colon C\rightarrow \Gamma'\subseteq \PP^2$ is a \emph{nodal} octic curve. In particular, $\pi$ is the normalization map and $\Gamma'$ has precisely $17={7\choose 2}-4$ nodes. 

\vskip 3pt

Assuming this claim for the moment, we set $n_i':=\pi_{\ell}(t_i)=\pi_{\ell}(x_i)=\pi_{\ell}(y_i)$, where $x_i, y_i\in C$ are defined via (\ref{def:xyt}).  We denote by $\bigl\{\mathpzc{o}_1, \ldots, \mathpzc{o}_9\}:=\mbox{Sing}(\Gamma')\setminus \{n_1', \ldots, n_8'\}$ the set of remaining nodes of $\Gamma'$. Let $E_i$ be the exceptional divisor at $\mathpzc{o}_i$ on the blow-up $\mbox{Bl}_9(\PP^2)$.

\vskip 4pt

\noindent {\bf Claim (ii):} One has an embedding
$\varphi:=\phi_{|4h-E_1-\cdots-E_9|}\colon \mbox{Bl}_9(\PP^2)\hookrightarrow \PP^5.$

\vskip 3pt

Via Proposition \ref{prop:T}, in order to establish claim (ii), it suffices to show that through the points $\mathpzc{o}_1, \ldots, \mathpzc{o}_9\in \PP^2$ there passes a unique smooth cubic curve.

\vskip 3pt

Assuming both claims (i) and (ii), we proceed with our proof. The map $\varphi$ is an embedding and from Proposition \ref{prop:T} its image $T\subseteq \PP^5$ is a projectively normal surface. We consider the image $\Gamma\subseteq \PP^5$ of the strict transform of $\Gamma'$ in $\mbox{Bl}_9(\PP^2)$  under the map $\varphi$. Then $\Gamma$ has nodes at the points $n_i:=\varphi(n_i')\in \PP^5$ for $i=1, \ldots, 8$ and is of degree
$$\mbox{deg}(\OO_{\Gamma}(1))=(8h-2E_1-\cdots-2E_9)(4h-E_1-\cdots-E_9)=14.$$ Comparing degrees, we conclude that $\Gamma\subseteq \PP^5$ is a quadratic section of $T$. Furthermore, we have a sequence of maps $C\rightarrow \Gamma \rightarrow \Gamma'$, showing that the smooth hyperelliptic curve $C$ is the normalization of $\Gamma$. Summarizing all this, the assignment
$$\vartheta\bigl((t_1, \ldots, t_8, \ell, C)\bigr):=[\Gamma, \omega_{\Gamma}(-1)]\in \mathfrak{Hyp}_{4,8},$$
where $\omega_{\Gamma}(-1)=\OO_{\Gamma}(4h-E_1-\cdots-E_9)$ is well-defined.

\vskip 5pt

We now show that each irreducible $8$-nodal curve $[\Gamma, L]\in \mathfrak{Hyp}_{4,8}$ which is general in any component of $\mathfrak{Hyp}_{4,8}$ appears this way. We fix such a pair and we may assume that $L\in W^2_8(\Gamma)$ is base point free. Let $\nu
\colon C\rightarrow \Gamma$ be the normalization map and denote by $\Gamma'$ the image of the map $\phi_{|L|}\colon \Gamma \rightarrow \PP^2$. Setting $\{n_1, \ldots, n_8\}=\mbox{Sing}(\Gamma)$, we denote by $\{x_i, y_i\}:=\nu^{-1}(n_i)$ the inverse images of the nodes of $\Gamma$ for $i=1, \ldots, 8$. Then $[C,\nu^*(L)]$ is a general point of the universal Picard variety $\mathcal{P}ic_{\mathcal{H}_4}^8$. Via Proposition \ref{pic:hyp} we may assume $C$ is embedded in the cubic scroll $Z\cong \mbox{Bl}_{\mathpzc{o}}(\PP^2)\subseteq \PP^4$ as a curve in the linear system $|6h-4E|$. Furthermore $\nu^*(L)=\OO_C(1)=\OO_C(2h-E)$, where, as usual, $E$ is the exceptional divisor
at the point $\mathpzc{o}$. Set $\pi:=\phi_{|L|}\circ \nu\colon C\rightarrow \PP^2$.

\vskip 3pt

Using the canonical isomorphism $H^0(Z, \OO_Z(1))\cong H^0(C, \OO_C(1))$, we set $$\ell:=\PP\Bigl(H^0(Z, \OO_Z(1)\bigr)/\pi^* H^0(\PP^2, \OO_{\PP^2}(1))^{\vee}\Bigr)\in \GG(1,4)$$ and consider the projection $\pi_{\ell}\colon \PP^4\dashrightarrow \PP^2$ with center $\ell$. Clearly,
$\pi_{\ell}$ is an  extension of the regular map $\pi:C\rightarrow \PP^2$. Note that $n_i':=\pi_{\ell}(x_i)=\pi_{\ell}(y_i)$.
We set $t_i:=\pi_{\ell}^{-1}(n_i')\setminus \{x_i, y_i\}\in Z$. Then clearly
$\vartheta\bigl((t_1,\ldots,t_8, \ell,C)\bigr)=[C,L]$, thus showing that $\vartheta$ is dominant.

\vskip 5pt

\noindent \emph{Proof of the claims (i) and (ii).} By degeneration we exhibit a point $\mathpzc{p}:=(t_1, \ldots, t_8, \ell, C)\in \mathcal{T}$, where $C$ is a \emph{reducible} nodal curve, such that $\vartheta(\mathpzc{p})$ is well-defined and both (i) and (ii) hold.

For a line $\ell$ in $\PP^4$, if again $\pi_{\ell}\colon \PP^4\dashrightarrow \PP^2$ is the projection with center $\ell$, the rational map
$$\xi:|4h-2E|\times \GG(1,4)\dashrightarrow \bigl(\PP^2\bigr)^{[8]}, \ \ \ \  \xi(B, \ell):=\mathrm{Sing}(\pi_{\ell}(B)),$$
is dominant. Therefore we can start with $8$ general points $\mathpzc{o}_1, \ldots, \mathpzc{o}_8\in \PP^2$ and choose a smooth genus
$2$ curve $B \in |4h-2E|$ on $Z$ and a line $\ell\subseteq \PP^4$, such that the image $B':=\pi_{\ell}(B)\subseteq \PP^2$ is a
sextic plane curve with nodes at $\mathpzc{o}_1, \ldots, \mathpzc{o}_8$ and no further singularities.

\vskip 3pt

The linear system $\bigl|\mathcal{I}_{\{\mathpzc{o}_1, \ldots, \mathpzc{o}_8\}/\PP^2}(3)\bigr|$ is a \emph{general} pencil of plane cubics. Its general member is smooth and its $12$ singular members are irreducible one-nodal rational curves with singularities disjoint from the set $\{\mathpzc{o}_1, \ldots, \mathpzc{o}_8\}$.  The plane cubics through $\mathpzc{o}_1, \ldots, \mathpzc{o}_8$
cut out the canonical linear system on $B$, that is,
$$ \pi_{\ell}^*\Bigl(\bigl|\mathcal{I}_{\{\mathpzc{o}_1, \ldots, \mathpzc{o}_8\}/\PP^2}(3)\bigr|\Bigr)=|\omega_B|.$$
This implies that the ninth remaining base point of the pencil  $\bigl|\mathcal{I}_{\{\mathpzc{o}_1, \ldots, \mathpzc{o}_8\}/\PP^2}(3)\bigr|$
does not lie on $B'$, since otherwise $B$ would have a pencil of degree one, hence $B$ would be rational.

\vskip 3pt

We now choose two general rulings $\ell_1$ and $\ell_2$ of $Z$, that is, $\ell_i\equiv h-E$, and set $F_i':=\pi_{\ell}(\ell_i)$.
Both $F_1'$ and $F_2'$ are lines in $\PP^2$ meeting in a point $\mathpzc{o}_9$. Furthermore, $\ell_i\cdot B=2$ and we set
\begin{equation}\label{def:Csing}
C:=B+\ell_1+\ell_2\in |6h-4E|,
\end{equation}
viewed as a nodal hyperelliptic curve of genus $4$. Note that both $\ell_1$ and $\ell_2$ meet $B$ in a pair of hyperelliptic conjugate points.
The image curve
$$\Gamma':=\pi_{\ell}(C)=B'+F_1'+F_2'\subseteq \PP^2$$
is a reducible nodal octic, where for $i=1,2$, the intersection $B'\cdot F_i'$ consists of $6$ nodes, namely the $2$ points in $\pi_{\ell}(B\cdot \ell_i)$, as well as $4$ further nodes on each $F_1'$ and $F_2'$ respectively.

\vskip 4pt

Since $\mathpzc{o}_9$ can be chosen freely in $\PP^2$, through the points $\mathpzc{o}_1, \ldots, \mathpzc{o}_9$ there passes a unique smooth cubic. Therefore, the map $\varphi:=\phi_{|4h-E_1-\cdots-E_9|}\colon \mbox{Bl}_{\mathpzc{o}_1, \ldots, \mathpzc{o}_9}(\PP^2)\hookrightarrow T\subseteq \PP^5$ is an embedding. The image $F_i\subseteq \PP^5$ of the strict transform in $\mbox{Bl}_{\mathpzc{o}_1, \ldots, \mathpzc{o}_9}(\PP^2)$ of $F_i'$ is a twisted cubic, whereas the image under $\varphi$ of the proper transform of $B'$ can be identified with the original smooth  genus $2$ curve $B$ embedded by the linear system $\bigl|\omega_B\otimes \pi_{\ell|B}^*\OO_{\PP^2}(1)\bigr|$.
The intersection $F_i\cdot B$ on $T$ is transverse and consists of the $6$ points in $\varphi(F_i'\cdot B')$, for $i=1,2$. Finally, $F_1$ and $F_2$ are disjoint. We consider the nodal curve

\begin{equation}\label{def:degg}
\Gamma:=B+F_1+F_2\subseteq \PP^5.
\end{equation}

\vskip 4pt

Set $\{t_1, \ldots, t_8\}:=\pi_{\ell}^{-1}\Bigl(\bigl(F_1'+F_2'\bigr)\cdot B'\setminus \pi_{\ell}\bigl((F_1+F_2)\cdot B\bigr)\Bigr)\setminus (B+\ell_1+\ell_2)\subseteq Z$. After choosing an ordering on this set of $8$ points, clearly $\mathpzc{p}:=(t_1, \ldots, t_8,\ell, C)\in \mathcal{T}$ and $\vartheta(\mathpzc{p})=[\Gamma, \omega_{\Gamma}(-1)]$, with $\OO_{\Gamma}(1)$ being defined via the embedding (\ref{def:degg}). The $8$ assigned nodes of $\Gamma$ as an element of $\mathfrak{Hyp}_{4,8}$ are the points in $(F_1+F_2)\cdot B$ that are \emph{not} the images of $(\ell_1+\ell_2)\cdot B$.  This completes the proof of both claims (i) and (ii).
\end{proof}

\begin{corollary}\label{cor:projnorm}
For a  general point $[\Gamma,L]\in \mathfrak{Hyp}_{4,8}$, the curve
$\phi_{\omega_{\Gamma}\otimes L^{\vee}}\colon \Gamma \hookrightarrow \PP^5$ is a projectively normal $8$-nodal curve of degree $14$.
\end{corollary}
\begin{proof}
Keeping the notation from the proof of Theorem \ref{thm:domin}, we consider the reducible nodal curve $\Gamma=F_1+F_2+B$ defined
by (\ref{def:degg}) and which appears as a quadratic section of the surface $T\subseteq \PP^5$. We have the following commutative diagram:
$$
\xymatrix{
\mbox{Sym}^2 H^0 (T, \OO_T(1))
\ar[r]^{\cong}  \ar[d]^{\mu_T} & \mbox{Sym}^2 H^0(\Gamma, \OO_{\Gamma}(1)) \ar[d]^{\mu_{\Gamma}} \\
H^0(T, \OO_T(2)) \ar[r] & H^0(\Gamma,\OO_{\Gamma}(2))
}
$$
The bottom map in this diagram is surjective. By Proposition \ref{prop:T} the surface $T$ is projectively normal, thus it follows that the same holds for $\Gamma$.
\end{proof}

\section{From scrolls of degree 9 to nodal hyperelliptic curves}

Let $\H_9$ denote the Hilbert scheme of degree $9$ scrolls $R\subseteq \PP^5$. The general point of $\H_9$ corresponds to a smooth degree $9$ scroll $R\subseteq \PP^5$. Following e.g. \cite[Lemma 1.5]{L}, one knows that $\H_9$ is smooth of dimension $h^0(R, N_{R/\PP^5})=59$. We denote by $\H_9^8$ the closure in $\H_9$ of the locus of scrolls having precisely $8$ (non-normal) nodes and no further singularities.  Using \cite[Proposition  0.4]{L} it follows that $\H_9^8$ is nonempty and has pure codimension $8$ inside $\H_9$, that is, $\mbox{dim}( \H_9^8)=51$. We introduce the parameter space of \emph{unparametrized} degree $9$  nodal scrolls
$$\mathfrak{H}_{\mathrm{scr}}:=\H_9^8/PGL(6).$$
Theorem \ref{isom}, coupled with results from \cite{L}, imply that $\hh$ is an irreducible variety of dimension $16=
\mbox{dim}( \H_9^8)-\mbox{dim } PGL(6)$.

\vskip 4pt

Each nodal scroll $[R]\in \mathfrak{H}_{\mathrm{scr}}$ is a projection $\pi\colon \PP^{10}\dashrightarrow \PP^5$ of a smooth degree $9$ scroll $$\FF_1:=\mbox{Bl}_{\mathpzc{o}}(\PP^2)\hookrightarrow R'\subseteq \PP^{10},$$ embedded by the linear system $\phi_{|5h-4E|}\colon \FF_1\hookrightarrow \PP^{10}$. Here $h$ is the pull-back of the line class under the morphism $\FF_1\rightarrow \PP^2$ and
$E$ denotes the exceptional divisor corresponding to the point $\mathpzc{o}\in \PP^2$. The rulings of $R'$ are the fibers of the morphism $\phi_{|h-E|}\colon R'  \rightarrow \PP^1$ and correspond to lines in $\PP^{10}$. The center of the projection $\pi$ is a $4$-plane $\Lambda \subseteq \PP^{10}$ which is $8$-secant to the secant variety $\mbox{Sec}(R')$ and the restriction $\pi\colon R'\rightarrow R$ of the projection map $\pi$ may be regarded as the normalization of $R$.  We denote by $\{n_1, \ldots, n_8\}\subseteq \PP^5$ the set of (non-normal) nodes of $R$ and by $\{x_i, y_i\}=\pi^{-1}(n_i)\subseteq R'$, for $i=1, \ldots, 8$. The projections of the rulings of $R'\subseteq \PP^{10}$ passing through $x_i$ and $y_i$ correspond to lines on $R\subseteq \PP^5$ meeting in the node $n_i$.

\begin{definition} We denote by $\mathcal{P}$ the moduli space of pairs $[R, \ell_1+\ell_2+\ell_3+\ell_4]$, where $[R]\in \mathfrak{H}_{\mathrm{scr}}$ and $\ell_1, \ldots, \ell_4\in \GG(1,5)$ are rulings of $R\subseteq \PP^5$, regarded as an unordered set.
\end{definition}
Over a nonempty open subset, the projection map $\mathcal{P}\rightarrow \mathfrak{H}_{\mathrm{scr}}$ sending $[R, \ell_1+\cdots+\ell_4]$ to $[R]$ is a $\PP^1\cong \mm_{0,4}/\mathfrak{S}_{4}$-bundle. To a general element $[R, \ell_1+\ell_2+\ell_3+\ell_4]\in \P$ (in which case we may assume that the rulings $\ell_i$ are disjoint from $\mbox{Sing}(R)$), we can associate a unique quadric $Q\subseteq \PP^5$ containing $\mbox{Sing}(R)$ and the rulings $\ell_1, \ldots, \ell_4$. The quadric $Q$ determines  a residual curve $\Gamma\subseteq \PP^5$ given by the relation
(\ref{def:gamma}), that is,
$$R\cdot Q=\ell_1+\ell_2+\ell_3+\ell_4+\Gamma.$$

Next we show that the assignment $[R]\mapsto[\Gamma]$ described by (\ref{def:gamma}) induces a well defined map $\chi\colon \mathcal{P}\dashrightarrow \mathfrak{Hyp}_{4,8}$. For our next result, recall that we have studied in Theorem \ref{thm:domin} the dominant morphism $\vartheta\colon \mathcal{T}\dashrightarrow \mathfrak{Hyp}_{4,8}$.

\begin{proposition}\label{well-def} For a general element $[R, \ell_1+\ell_2+\ell_3+\ell_4]\in \hh$, the  curve $\Gamma\subseteq \PP^5$ has nodes at $n_1, \ldots, n_8$ and no further singularities. Its normalization $C:=\pi^{-1}(\Gamma)\subseteq R'$ is a smooth hyperelliptic curve of genus $4$.
\end{proposition}
\begin{proof}
Recall that $\pi\colon R'\rightarrow R$ is the normalization map and set $\ell_i':=\pi^{-1}(\ell_i)$. Since $\ell_i'$ is a ruling of $R'$ we have that $\ell_i\equiv h-E$, where we have identified $\FF_1$ and $R'$. From (\ref{def:gamma}) we then obtain  $C\equiv 6h-4E$. Furthermore,  $C\cdot (h-E)=2$, that is, $\OO_C(h-E)\in W^1_2(C)$ is the hyperelliptic linear system on $C$. For a general choice of the rulings, we have $\ell_i\cap \{n_1, \ldots, n_8\}=\emptyset$, therefore from (\ref{def:gamma}) it follows that $n_i\in \Gamma$ and hence $x_i, y_i\in C$.

\vskip 3pt

We now show that $\chi$ is well defined and in fact $\mbox{Im}(\chi)\cap \mbox{Im}(\vartheta)\neq \emptyset$.  To that end we use a further degeneration inside the linear system $|6h-4E|$ on the cubic scroll $Z\subseteq \PP^4$. We start with two general curves $\Gamma_1, \Gamma_2\in |3h-2E|$ on $Z$. Both $\Gamma_1$ and $\Gamma_2$ are smooth rational curves meeting in $5=(3h-2E)^2$ points, which we call $v_1, \ldots, v_5\in Z$. The union $\Gamma_1\cup \Gamma_2$ is a stable hyperelliptic curve of genus $4$ and the hyperelliptic involution interchanges $\Gamma_1$ and $\Gamma_2$. Precisely, if $u:=\phi_{|h-E|}\colon Z\rightarrow \PP^1$ is the fibration given by the rulings of $Z$, then $\iota\colon \Gamma_1 \rightarrow \Gamma_2$ is the isomorphism  given by $\iota(u^{-1}(t)\cdot \Gamma_1)=u^{-1}(t)\cdot \Gamma_2$, for every $t\in \PP^1$, with the map $(\iota, \iota^{-1})\colon \Gamma_1\cup \Gamma_2\rightarrow \Gamma_2\cup \Gamma_1$  being the hyperelliptic involution.

\vskip 3pt

We pick a general line $\ell\in \GG(1,4)$ and consider the degree $3$ map $\pi_{\ell}\colon Z\rightarrow \PP^2$ obtained by restricting the projection map. Set $\Gamma_i':=\pi_{\ell}(\Gamma_i)\subseteq \PP^2$, for $i=1,2$. Both $\Gamma_1'$ and $\Gamma_2'$ are $3$-nodal plane quartics and we set $\mathrm{Sing}(\Gamma_1')=\{\mathpzc{o}_1, \mathpzc{o}_2, \mathpzc{o}_3\}$ and $\mathrm{Sing}(\Gamma_2')=\{\mathpzc{o}_4, \mathpzc{o}_5, \mathpzc{o}_6\}$ respectively. The curves $\Gamma_1'$ and $\Gamma_2'$ meet transversally at $16$ points, namely $n_{i+8}':=\pi_{\ell}(v_i)$ for $i=1, \ldots, 5$ and $11$ further nodes which we partition into two groups we denote by $\{\mathpzc{o}_7, \mathpzc{o}_8, \mathpzc{o}_9\}$ and $\{n_1', \ldots, n_{8}'\}$ respectively. It is now straightforward to check that $\Gamma_1$ and $\Gamma_2$ can chosen in such a way that through $\mathpzc{o}_1, \ldots, \mathpzc{o}_9$ there passes a unique smooth cubic curve. Recalling that $E_i$ is the exceptional divisor corresponding to the point $\mathpzc{o}_i$ for $i=1,\ldots, 9$, we apply Proposition \ref{prop:T} and obtain that the map
$$\varphi:=\phi_{|4h-E_1-\cdots-E_9|}\colon \mathrm{Bl}_{9}(\PP^2)\hookrightarrow \PP^5$$
is an embedding having as image a projectively normal surface $T\subseteq \PP^5$ of degree $7$. The image under $\varphi$ of the strict transform of $\Gamma_1'\cup \Gamma_2'$ in $\mbox{Bl}_9(\PP^2)$ is the stable curve
\begin{equation}\label{def:redG}
\Gamma=\Gamma_1+\Gamma_2 \subseteq T\subseteq \PP^5,
\end{equation}
where $\Gamma_1$ and $\Gamma_2$ are smooth rational curves of degree $7$ meeting at the points $n_i:=\varphi(n_i')$, for $i=1,\ldots, 13$.
The $8$-nodal scroll associated via (\ref{def:gamma}) to the curve $\Gamma$ described by (\ref{def:redG}) is then
$$R:=\bigcup_{x\in \Gamma_1} \bigl\langle x, \iota(x)\bigr\rangle\subseteq \PP^5.$$
This is a scroll of degree $9$ which has nodes at $n_1, \ldots, n_8$ (and not at $n_9, \ldots, n_{13}$, which are fixed by the hyperelliptic involution on $\Gamma_1\cup \Gamma_2$). One checks either directly or with \emph{Macaulay} that $R$ has no further singularities. This shows that
$[\Gamma, \omega_{\Gamma}(-1)]\in \mbox{Im}(\vartheta)\cap \mbox{Im}(\chi)$.
\end{proof}

\vskip 4pt

\begin{remark} The construction described in Proposition \ref{well-def} provides an alternative proof of the result in \cite{L} stating that the Hilbert scheme $\H_9^8$ of $8$-nodal scrolls of degree $9$ is nonempty.
\end{remark}

Therefore we have a well-defined rational map
$$\chi\colon \P\dashrightarrow \mathfrak{Hyp}_{4,8}, \ \ \ \ \chi\bigl([R, \ell_1+\ell_2+\ell_3+\ell_4]\bigr):=[\Gamma, \omega_{\Gamma}(-1)].$$
As discussed in the previous section, since $\OO_{\Gamma}(2)$ is non-special, applying  Riemann-Roch we find $h^0(\Gamma, \OO_{\Gamma}(2))=2\mbox{deg}(\Gamma)+1-p_a(\Gamma)=17$. Since the image of $\chi$ intersects the image of the dominant map $\vartheta\colon \mathcal{T}\dashrightarrow \mathfrak{Hyp}_{4,8}$, applying Corollary \ref{cor:projnorm} we may also assume that $\Gamma\subseteq \PP^5$ is projectively normal, hence $h^0\bigl(\PP^5, \mathcal{I}_{\Gamma/\PP^5}(2)\bigr)=4$.  Furthermore, the last part of Proposition \ref{prop:T} yields that  base locus $\mathrm{Bs}
\ \bigl|\mathcal{I}_{\Gamma/\PP^5}(2)\bigr|$ is a nodal curve $Y$ of degree $16$ containing $\Gamma$ as a component, that is,
\begin{equation}\label{YB}
Y=\Gamma+B,
\end{equation}
 with $\omega_Y=\OO_Y(2)$.
Furthermore, the curves $B$ and $\Gamma$ meet transversally. Using \cite[Example 9.1.12]{Ful}, we have the formulas
$$p_a(\Gamma)-p_a(B)=\frac{1}{2}\cdot (8-6)\bigl(\mbox{deg}(\Gamma)-\mbox{deg}(B)\bigr)=12, \ \ \Gamma\cdot B=2\mbox{deg}(\Gamma)+2-2p_a(\Gamma)=6,$$ hence $B$ is a smooth conic $6$-secant to $\Gamma$. We introduce the $2$-plane spanned by $B$
$$\Pi:=\langle B\rangle\subseteq \PP^5.$$
Viewing $B\cdot \Gamma$ as a degree $6$ divisor on $\Gamma$  disjoint from the nodes $n_1, \ldots, n_8$, we have
\begin{equation}\label{gammaem}
\omega_{\Gamma}(B\cdot \Gamma)\cong \OO_{\Gamma}(2).
\end{equation}

\begin{proposition}
There exists a $3$-dimensional linear system $V\subseteq H^0\bigl(\PP^5, \mathcal{I}_{\Gamma/\PP^5}(2)\bigr)$ containing $\Pi$.
\end{proposition}
\begin{proof} We pick a general point $\mathpzc{r} \in \Pi\setminus B$. Then for a quadric $Z\in  H^0\bigl(\PP^5, \mathcal{I}_{\Gamma/\PP^5}(2)\bigr)$ one has that  $\Pi\subseteq Z$ if and only if $\mathpzc{r} \in Z$. Indeed, if $\mathpzc{r} \in Z$, then the restriction of $Z$ to $\Pi$ already contains $B\cup \{\mathpzc{r}\}$, therefore $\Pi\subseteq Z$. Since containing the fixed point $\mathpzc{r}$ imposes one condition on $|\mathcal{I}_{\Gamma/\PP^5}(2)|$, the conclusion follows.
\end{proof}

We now introduce the surface $T\subseteq \PP^5$, defined as the residual surface to $\Pi$ in the complete intersection
(\ref{defi:T}), that is,
$$\mathrm{Bs }\ |V|=\Pi+T.$$

Thus $T$ is a degree $7$ surface in $\PP^5$ lying on three quadrics whose intersection contains a $2$-plane. Such surfaces are classified in \cite{Io} and there are five possible families. But the geometric situation at hand helps us show that $T$ is the surface described in Proposition \ref{prop:T}.  Since $\Gamma$ is nondegenerate in $\PP^5$, in particular $\Gamma\nsubseteq \Pi$, hence  $\Gamma\subseteq T$.  It follows that $\Gamma$ is the intersection of $T$ with one of the quadrics from $H^0\bigl(\PP^5, \mathcal{I}_{\Gamma/\PP^5}(2)\bigr)\setminus V$. Since the intersection $\Gamma\cap B$ is transverse, one has $n_i\notin B$, and hence $n_i\in \PP^5\setminus \Pi$. We set $n_i':=p(n_i)\in \PP^2$ for $i=1, \ldots, 8$, where
\begin{equation}\label{def:proj}
p=p_{\Pi}\colon \PP^5\dashrightarrow \PP^2
\end{equation}
is the projection with center the $2$-plane $\Pi$.

\begin{proposition}\label{hyp3}
The image curve $\Gamma':=p(\Gamma)\subseteq \PP^2$ is a nodal plane curve of genus $8$ with nodes at $n_1', \ldots, n_8'$, as well as at  further  $9$ unspecified points.
\end{proposition}
\begin{proof}
Set $\Gamma\cdot B=\mathpzc{r}_1+\cdots+\mathpzc{r}_6$ viewed as a divisor of degree $6$ on $\Gamma$. Since for any quadric $q\in H^0\bigl(\PP^5, \mathcal{I}_{\Gamma/\PP^5}(2)\bigr)\setminus V$  one  has  $q\cdot \Pi=Q$, it follows  that $\Gamma \cap \Pi=\{\mathpzc{r}_1, \ldots, \mathpzc{r}_6\}$.

The restriction $p_{|\Gamma}\colon \Gamma\dashrightarrow \PP^2$ of the projection map $p$ defined by (\ref{def:proj}) has thus the divisor $\mathpzc{r}_1+\cdots+\mathpzc{r}_6$ as its base locus. Removing this base locus we obtain a regular map $p\colon \Gamma\rightarrow \Gamma'\subseteq \PP^2$ given by the linear system  $$\bigl|\OO_{\Gamma}(1)(-\mathpzc{r}_1-\cdots-\mathpzc{r}_6))\bigr| \cong| \omega_{\Gamma}(-1) |\in W^2_8(\Gamma).$$ Thus $\Gamma'$ is a plane octic curve. Since its normalization is the genus $4$ curve $C$, it has $17$ nodes, which fall into two groups, namely the $8$ nodes $\{n_1', \ldots, n_8'\}$ and the rest, which we denote by $\{\mathpzc{o}_1, \ldots, \mathpzc{o}_9\}$.  We remark that $\Gamma\subseteq \PP^5$ can be recovered from such an octic curve $\Gamma'$.  Indeed, keeping the notation above, we blow-up $\PP^2$ at the nine points $\mathpzc{o}_1, \ldots, \mathpzc{o}_9\in \PP^2$, then consider the regular map
$\varphi:=\phi_{|4h-E_1-\cdots-E_9|}\colon \mathrm{Bl}_{9}(\PP^2)\rightarrow \PP^5.$ Using Proposition \ref{prop:T}, we may assume $\varphi$ is indeed an embedding.
The image of the strict transform of $\Gamma'$ under $\varphi$ is precisely the $8$-nodal curve $\Gamma$.
\end{proof}

\vskip 3pt

\begin{remark}
The conic $B$ defined in (\ref{YB}) is the intersection of the quadric $q$ with the $2$-plane $\Pi$, whereas the cycle $\Gamma\cdot B$ of length $6$ is precisely the intersection cycle $\Gamma\cdot J$ on the smooth surface $T$, where $J\in|3h-E_1-\cdots-E_9|$.
\end{remark}

We can now finish the proof of Theorem \ref{isom} and \ref{unir}.

\vskip 4pt

\noindent \emph{Proof of Theorem \ref{isom}.} It suffices to observe  that the map $\chi$ is generically injective. Indeed, for the $8$-nodal curve $\Gamma\subseteq \PP^5$ given by (\ref{def:gamma}) and having a hyperelliptic normalization $\pi\colon C\rightarrow \Gamma$ of genus $4$, we recover the scroll as the union of the lines $\bigl\langle \pi(x), \pi(\iota(x))\bigr\rangle\subseteq \PP^5$ as $x\in C$ varies.
Here $\iota$ denotes the hyperelliptic involution of $C$. But then the quadric $Q\subseteq \PP^5$ such that $\Gamma\subseteq R\cdot Q$ is also determined, which also leads to the unordered collection $\ell_1+\ell_2+\ell_3+\ell_4$ of rulings of $R$ such that the relation (\ref{def:gamma}) holds.
\hfill $\Box$

\end{document}